\documentclass[reqno]{amsart}
\usepackage[utf8]{inputenc}
\usepackage{amssymb, amsmath, amsthm, array, color, enumerate, mathabx, mathrsfs}

\usepackage{bm}

\usepackage{pdfpages}

\usepackage{tcolorbox}

\numberwithin{equation}{section}

\newcommand{\R}{\mathbb{R}}

\newcommand{\Z}{\mathbb{Z}}
\newcommand{\N}{\mathbb{N}}

\renewcommand{\H}{\mathbb{H}}

\newcommand{\bdot}{\bm{\cdot}}

\newcommand{\ud}{\mathrm{d}}

\newcommand{\bg}{\mathbf{g}}

\newcommand{\fg}{\mathfrak{g}}

\def\inn#1#2{\langle#1,#2\rangle}

\theoremstyle{plain}

\newtheorem{theorem}{Theorem}[section]
\newtheorem{lemma}[theorem]{Lemma}
\newtheorem{proposition}[theorem]{Proposition}

\theoremstyle{definition}
\newtheorem{definition}[theorem]{Definition}

\begin{document}

\title[Lacunary maximal functions]{Lacunary maximal functions on homogeneous groups}

\begin{abstract}
    We observe that classical arguments of Ricci--Stein can be used to prove $L^p$ bounds for maximal functions associated to lacunary dilates of a fixed measure in the setting of homogenous groups. This recovers some recent results on averages over Kor\'anyi spheres and horizontal spherical averages of a type introduced by Nevo--Thangavelu. Moreover, the main theorem applies much more broadly and we explore its consequences through a variety of explicit examples. 
\end{abstract}

\author[A. Govindan Sheri]{Aswin Govindan Sheri}
\address{School of Mathematics, James Clerk Maxwell Building, The King's Buildings, Peter Guthrie Tait Road, Edinburgh, EH9 3FD, UK.}
\email{a.govindan-sheri@sms.ed.ac.uk}

\author[J. Hickman]{Jonathan Hickman}
\address{School of Mathematics, James Clerk Maxwell Building, The King's Buildings, Peter Guthrie Tait Road, Edinburgh, EH9 3FD, UK.}
\email{jonathan.hickman@ed.ac.uk}

\author[J. Wright]{James Wright}
\address{School of Mathematics, James Clerk Maxwell Building, The King's Buildings, Peter Guthrie Tait Road, Edinburgh, EH9 3FD, UK.}
\email{j.r.wright@ed.ac.uk}

\date{}

\maketitle




\section{Main result and examples}




\subsection{Introduction}\label{subsec: intro} Recently, there has been some interest in the $L^p$ mapping properties of geometric maximal operators over homogeneous groups $G$. Here we consider operators $f \mapsto \sup_k |f \ast \sigma_k|$ acting on functions $f \in C_c(G)$, where $\sigma$ is a measure supported on some submanifold $S$ of $G$, the $\sigma_k$ are (automorphic) lacunary dilates of $\sigma$ and the convolution is taken with respect to the group operation. In the case of the Heisenberg groups $\H^m$, prototypical examples arise by taking $S$ (in exponential coordinates) to be:
\begin{enumerate}[1)]
    \item The Kor\'anyi sphere $S := \{(x,u) \in \H^m : |x|^4 + |u|^2 = 1\}$: see \cite{GT2021, Srivastava2022};
    \item The horizontal (euclidean) sphere $S := S^{2m-1} \times \{0\}$: see \cite{BHRT2021, RSS2022}.
\end{enumerate}
We shall discuss these examples in more detail in \S\ref{subsec: examples} below.\footnote{In many recent works, including those cited above, the authors were also (and often primarily) interested in proving sparse bounds for the associated maximal operators. Here we are only interested in the $L^p$ theory.}

Here we observe that classical arguments of Ricci--Stein~\cite{RS1988} lead to a very general and robust $L^p$ theory for such lacunary maximal functions. These methods were developed to study boundedness properties of classes of singular integrals and also maximal operators along homogeneous submanifolds (see \cite[Lemma 4.2]{RS1988}); it is therefore unsurprising that they are useful in the present context. Nevertheless, given the renewed interest in these problems, it appears timely to revisit the approach of~\cite{RS1988} and give a clear presentation of its implications.




\subsection{Setup and main results} Let $(G,\,\bdot\,)$ be a homogenous group with family of automorphic dilations $(\delta_t)_{t > 0}$; we recall the relevant definitions in \S\ref{subsec: basic concepts} below. Given a finite, compactly supported Borel measure $\sigma$ on $G$ and $k \in \Z$, define the $2^k$-dilate $\sigma_k$ of $\sigma$ and the reflection $\tilde{\sigma}$ of $\sigma$ by
\begin{equation*}
    \langle \sigma_k,\phi \rangle := \int_{G} \phi \circ \delta_{2^k} (x)\,\ud\sigma(x) \quad \textrm{and} \quad\inn{\tilde{\sigma}}{\phi} := \int_{G} \phi(x^{-1})\,\ud\sigma(x)  \qquad \textrm{for $\phi \in C_c(G)$.}
\end{equation*}
For any such measure $\sigma$ we consider the associated averaging operator  
\begin{equation*}
    A[\sigma]f(x) := f \ast \sigma(x) = \int_{G} f(x \bdot y^{-1}) \,\ud\sigma(y)  \qquad \textrm{ for $x \in G$,}
\end{equation*}
and lacunary maximal function
\begin{equation*}
    M[\sigma]f(x) := \sup_{ k\in \Z}|A[\sigma_k]f(x)| \qquad \textrm{for $x \in G$,} 
\end{equation*}
defined initially for $f \in C_c(G)$. We are interested in establishing conditions on $\sigma$ which ensure $M[\sigma]$ is bounded on $L^p(G)$ for all $1 < p \leq \infty$. 

Following Ricci--Stein~\cite{RS1988} (and also \cite{Christ1985}), we consider interated convolution products of $\sigma$. For $N \in \N_0$, we define the $N$th convolution product $\sigma^{(N)}$ recursively: starting with $\sigma^{(0)} = \sigma$, we take 
\begin{equation}\label{eq: conv prod}
    \sigma^{(N)} := \begin{cases} \sigma^{(N-1)} \ast \tilde{\sigma} & \textrm{if  $N$ is odd},\\
      \sigma^{(N - 1)} \ast \sigma & \text{if $N \geq 2$ is even.} \end{cases}
\end{equation}
Our main hypothesis is framed in terms of the regularity properties of the $\sigma^{(N)}$ for large $N$.

\begin{definition}[Curvature assumption]\label{def: CA}
Let $\sigma$ be a finite Borel measure on $G$. We say that $\sigma$ satisfies $(\mathrm{CA})$ if the following hold:
\begin{enumerate}[i)]
    \item There exists $N > 0$ such that $\sigma^{(N)}$ is absolutely continuous with respect to the Haar measure on $G$;
    \item If $h$ denotes the Radon--Nikodym derivative of $\sigma^{(N)}$ with respect to the Haar measure, then there exists $\gamma > 0$ and $C_{\sigma} > 0$ such that
\begin{equation*}
    \int_{G}|h(x\bdot y^{-1}) - h(x)| \,\ud x + \int_{G}|h(y^{-1}\bdot x) - h(x)|\,\ud x \leq C_{\sigma} |y|_G^{\gamma} \quad \text{for $y \in G$.}
\end{equation*}
\end{enumerate}
Here $|\,\cdot\,|_G$ denotes a choice of homogeneous norm on $G$; see \S\ref{subsec: basic concepts} for the definition.\footnote{The use of homogeneous norm is for convenience; here one could equally replace $|y|_G$ with the usual euclidean norm $|y|$.}
\end{definition}

With the above definition, our main theorem reads as follows.

\begin{theorem}\label{thm: main}
Let $G$ be a homogeneous group and $\sigma$ be a finite, compactly supported Borel measure on $G$. If $\sigma$ satisfies $(\mathrm{CA})$, then $M[\sigma]$ is bounded on $L^p(G)$ for all $1 < p \leq \infty$.
\end{theorem}

Theorem~\ref{thm: main} recovers many results in the literature. In the euclidean case it is equivalent to a classical result of Duoandikoetxea and Rubio de Francia \cite{DRdF1986} concerning maximal estimates under Fourier decay hypotheses on $\sigma$. For Heisenberg groups, it implies the $L^p$ boundedness of the spherical averages (over both Kor\'anyi and horizontal spheres) discussed in \S\ref{subsec: intro}. It also implies \cite[Lemma 4.2]{RS1988}, which concerns maximal functions along homogeneous submanifolds of $G$. However, the setup described above is very robust and there is a wealth of additional examples, some of which we describe in the following subsection.

As mentioned previously, the proof of Theorem~\ref{thm: main} is heavily based on the methods of \cite{RS1988}, which in turn were partially inspired by \cite{Christ1985}. Key ingredients are iterated $T^*T$ arguments, which allow us to access and exploit the condition $(\mathrm{CA})$, and Calder\'{o}n--Zygmund theory adapted to the homogeneous group setting.




\subsection{Examples}\label{subsec: examples} The curvature assumption from Definition~\ref{def: CA} applies to a broad class of measures and, consequently, Theorem~\ref{thm: main} unifies and dramatically extends a number of results in the literature. Here we briefly describe some representative examples of such measures; we return to discuss these and other examples in more detail in \S\ref{sec: analysis examples}.




\subsubsection*{1) The Euclidean case} Suppose $\sigma$ is a finite, compactly supported Borel measure on $\R^n$. It is not difficult to show $(\mathrm{CA})$ holds in this case if and only if there exists some $\kappa > 0$ and a constant $C_{\sigma} \geq 1$ such that
\begin{equation}\label{eq: Fourier decay}
    |\widehat{\sigma}(\xi)| \leq C_{\sigma} (1 + |\xi|)^{-\kappa} \quad \textrm{for all $\xi \in \widehat{\R}^n$,} \quad \textrm{where} \quad \widehat{\sigma}(\xi) := \int_{\R^n} e^{-2 \pi i x \cdot \xi} \,\ud \sigma(x).
\end{equation}
Thus, Theorem~\ref{thm: main} recovers a classical result of \cite{DRdF1986}. Prototypical examples of measures satisfying \eqref{eq: Fourier decay} are given by smooth densities supported on finite-type submanifolds of $\R^n$: see, for instance, \cite[Chapter VIII, \S3.2]{Stein_book}. There are also many interesting `fractal' measures. The literature is too vast to survey here, but we mention some recent representative examples: measures arising from the theory of diophantine approximation \cite{FH2023}, realisations of random processes \cite{FS2018} and also various non-constructive examples exhibiting interesting dimensional properties \cite{Korner2011}.




\subsubsection*{2) The Kor\'anyi sphere and extensions} For $m \geq 1$ let $J$ denote the symplectic form on $\R^{2m}$ given by $x^{\top}Jy = \tfrac{1}{2}\sum_{j=1}^m (x_jy_{m+j} - x_{m+j}y_j)$ for $x$, $y \in \R^{2m}$. Consider the Heisenberg group $\H^m$, which we identify with $\R^{2m+1}$ endowed with the non-commutative group operation $(x,u) \bdot (y, v) := (x+y, u + v + x^{\top}Jy)$. The \textit{Kor\'anyi sphere} is then defined to be the set
\begin{equation*}
    S := \{ (x,u) \in \R^{2m} \times \R : |(x,u)|_{\H^m} = 1\} \quad \textrm{where} \quad |(x,u)|_{\H^m} := \big(|x|^4 + |u|^2 \big)^{1/4};
\end{equation*}
that is, $S$ is the unit sphere with respect to the \textit{Kor\'anyi norm} $|\,\cdot\,|_{\H^m}$. We take $\sigma$ to be the normalised surface measure on $S$ induced by the Lebesgue measure on the ambient space $\R^{2m+1}$. 

It is not difficult to show $\sigma$ satisfies $(\mathrm{CA})$, and therefore Theorem~\ref{thm: main} implies the associated lacunary maximal operator $M[\sigma]$ is bounded for all $1 < p \leq \infty$. For $n \geq 2$, this recovers Theorem 1.2 of \cite{GT2021} (see also \cite{Srivastava2022}). However, here the setup can be generalised considerably. For instance, one may replace the Heisenberg group with any graded Lie group $G$ and the Kor\'anyi sphere $S$ with any member of a broad class of convex sets lying in $G$. More precisely, given any analytic submanifold $\Sigma$ of the Lie algebra $\fg$ formed by the boundary of an open, bounded convex set, we can consider averages over $S := \exp(\Sigma) \subset G$. We discuss the details of these extensions in \S\ref{subsec: Koranyi} below. 



\subsubsection*{3) Horizontal spheres and extensions} Returning to the Heisenberg group $\H^m$, now consider the normalised surface measure $\sigma$ on the (euclidean) sphere $S^{2m-1} \times \{0\} \subseteq \R^{2m+1}$. In this case the operators
\begin{equation}\label{eq: horizontal spherical av}
    A[\sigma_k] f(x,u) = \int_{S^{2m-1}} f\big(x-2^k y, u - 2^k x^{\top} J y\big) \,\ud \sigma(y), \qquad (x,u) \in \R^{2m+1},
\end{equation}
take averages over ellipsoids lying in the horizontal distribution on $\H^m$. Averages of this kind were first considered by Nevo--Thangavelu~\cite{NT1997} (in the context of the `full' maximal function defined with respect to a continuum of dilates) and have been extensively studied: see, for instance, \cite{MS2004, NT2004, ACPS2021, PS2021, Bentsen2022, BGHS2022, RSS2022, LL2023}.

As before $\sigma$ satisfies $(\mathrm{CA})$ and we obtain the boundedness of $M[\sigma]$. This recovers Theorem 1.1 of \cite{BHRT2021} (see also \cite{RSS2022}, where the $n = 1$ case and extensions to M\'etivier groups are considered) although, as remarked in \cite{RSS2022}, such bounds can be directly deduced from earlier work such as \cite{MS2004}. However, once again the setup can be generalised considerably. In this case it is natural to work in a stratified group $G$ (we recall the relevant definitions in \S\ref{subsec: graded and stratified groups}), so that the Lie algebra $\fg = \bigoplus_{j=1}^{\infty} V_j$ is graded and $V_1$ is a vector subspace of dimension $d$ which generates $\fg$. We consider the (euclidean) unit sphere $S^{d-1}$ in $V_1$, which we map into $G$ via the exponential map. We then let $\sigma$ be the normalised surface measure on this sphere in $G$. The curvature assumption $(\mathrm{CA})$ continues to hold at this level of generality. Moreover, we may further replace the sphere $S^{d-1}$ in $V_1$ with some other surface, possibly of lower dimension. We say that an analytic submanifold $\Sigma$ of $V_1$ is \textit{non-degenerate} if it is not contained in a proper affine subspace of $V_1$. With this definition, $(\mathrm{CA})$ continues to hold for any smooth, compactly supported density on $S := \exp(\Sigma)$ whenever $\Sigma \subset V_1$ is a non-degenerate analytic submanifold.




\subsubsection*{4) Tilted spheres and extensions} An interesting variant of the previous example arises from `tilting' the sphere $S^{2m-1}$ in \eqref{eq: horizontal spherical av}. Let $v \in \R^{2m}$ and now consider the averaging operators
\begin{equation*}
    A[\sigma^v_k] f(x,u) = \int_{S^{2m-1}} f\big(x-2^k y, u - 2^{2k}\inn{v}{y} - 2^k x^{\top} J y\big) \,\ud \sigma(y), \qquad (x,u) \in \R^{2m+1},
\end{equation*}
which correspond to Heisenberg convolution with a dilates of a suitably `tilted' variant $\sigma^v$ of the measure $\sigma$ on $S^{2m-1} \times \{0\}$. Averages of this form, along with natural extensions to the class of M\'etivier groups, have been considered in a number of works~\cite{MS2004, ACPS2021, RSS2022}.

The tilted measures $\sigma^v$ also satisfy $(\mathrm{CA})$ and we again obtain the boundedness of $M[\sigma]$. The results also extend to general stratified groups and analytic submanifolds $\Sigma \subseteq V_1$ under the non-degeneracy hypothesis described in 3). However, we may go further. Indeed, we may replace the tilted sphere with any analytic submanifold $\Sigma \subset \fg$ with the property that the projection $\Pi_1(\Sigma)$ onto $V_1$ is not contained in any proper affine subspace of $V_1$. This includes the examples discussed here and in 3) above, but it also includes examples which are not associated to some $d$-plane distribution in $G$, such as averages over group translates of the moment curve $(t,t^2, \dots, t^n)$ in $G$. We discuss the details of these extensions in \S\ref{subsec: horizontal} below.




\subsection*{Notational conventions} 

Given a list of objects $L$ and real numbers $A$, $B \geq 0$, we write $A \lesssim_L B$ or $B \gtrsim_L A$ to indicate $A \leq C_L B$ for some constant $C_L$ which depends only on items in the list $L$ and the underlying choice of group $G$.\footnote{We also suppress the dependence on various constructions associated to $G$ such as a choice of homogeneous norm, basis $\{X_j^R\}_{j=1}^n$ of right-invariant vector fields, and so on.} We write $A \sim_L B$ to indicate $A \lesssim_L B$ and $B \lesssim_L A$. 




\subsection*{Acknowledgements} The first and second authors thank David Beltran and Leonardo Tolomeo for a variety of helpful comments and suggestions regarding this project. The first author was supported by The Maxwell Institute Graduate School in Analysis and its Applications, a Centre for Doctoral Training funded by the UK Engineering and Physical Sciences Research Council (grant EP/L016508/01), the Scottish Funding Council, Heriot--Watt University, and the University of Edinburgh.




\section{Proof of Theorem~\ref{thm: main}}




\subsection{Homogeneous groups}\label{subsec: basic concepts} We begin by reviewing some basic concepts from the theory of homogeneous groups relevant to the proof. Much of the material presented here can be found in Chapter 1 of the classical treatise \cite{FS_book}. 

Let $G$ be a connected, simply connected nilpotent Lie group with Lie algebra $\fg$; we shall always tacitly assume $\fg$ is real and finite-dimensional. Under these hypotheses, the exponential map $\exp \colon \fg \to G$ is a (globally defined) diffeomorphism. This allows us to identify $G$ with $\R^n$ endowed with a group operation $(x, y) \mapsto x \bdot y$. Furthermore, this group operation is a polynomial mapping and the usual Lebesgue measure on $\R^n$ is a left and right-invariant Haar measure for $G$. We refer to the dimension $n$ as the \textit{topological dimension} of $G$, and sometimes denote this by $\dim G$. 

If $\fg$ is a Lie algebra, then $(\delta_t)_{t > 0}$ is a \textit{family of dilations on $\fg$} if each $\delta_t \colon \fg \to \fg$ is an algebra automorphism\footnote{So that $\delta_t \colon \fg \to \fg$ is a linear map and $\delta_t[X,Y]_{\fg} = [\delta_t X, \delta_t Y]_{\fg}$ for all $X, Y \in \fg$, where $[\,\cdot\,,\,\cdot\,]_{\fg}$ denotes the Lie bracket on $\fg$.} and there exists some diagonalisable linear operator $\Lambda$ on $\fg$ with positive eigenvalues such that $\delta_t := \exp(\Lambda \log t)$ for all $t > 0$. In particular, $\delta_{st} = \delta_s \circ \delta_t$ for all $s$, $t >0$. In this case, $\Lambda$ is called the \textit{dilation matrix}.

A \textit{homogeneous group} is a connected, simply connected nilpotent Lie group $G$ such that the Lie algebra $\fg$ is endowed with a family of dilations $(\delta_t)_{t > 0}$. In this case, the maps $\exp \circ \,\delta_t\, \circ \exp^{-1} \colon G \to G$ are group automorphisms, which are referred to as \textit{dilations of $G$} and are also denoted by $\delta_t$. 

Henceforth, let $G$ be a homogenous group $G$. The \textit{homogeneous dimension of $G$} is the quantity
\begin{equation*}
    Q := \mathrm{trace}(\Lambda) = \sum_{j=1}^n \lambda_j,
\end{equation*}
where $\Lambda$ is the dilation matrix and the $\lambda_j > 0$ are the eigenvalues of $\Lambda$ listed with multiplicity. The Haar measure on $G$ then satisfies the important scaling property
\begin{equation*}
    t^{-Q} \int_G f \circ \delta_{t^{-1}}(x)\,\ud x = \int_G f(x)\,\ud x \qquad \textrm{for all $f \in L^1(G)$ and $t > 0$.}
\end{equation*}

A \textit{homogeneous norm} on $G$ is a continuous map $|\,\cdot\,|_G \colon G \to [0, \infty)$ which is $C^{\infty}$ on $G \setminus \{0\}$ and satisfies:
\begin{enumerate}[a)]
    \item $|x^{-1}|_G = |x|_G$ and $|\delta_tx|_G = t|x|_G$ for all $x \in G$ and $t > 0$;
    \item $|x|_G = 0$ if and only if $x = 0$.
\end{enumerate}
Henceforth we shall assume $|\,\cdot\,|_G$ is some fixed choice of homogeneous norm on $G$; it is easy to see that there always exists at least one such norm. Furthermore, one can show that there exists a constant $C_G \geq 1$ such that 
\begin{equation}\label{eq: quasi triangle}
    |x\bdot y|_G \leq C_G\big(|x|_G + |y|_G\big) \qquad \textrm{for all $x$, $y \in G$.}
\end{equation}

Finally, we fix $\{X_1, \dots, X_n\} \subset \fg$ an orthogonal basis of eigenvectors for the dilation matrix $\Lambda$, so that $X_j$ has eigenvalue $\lambda_j$ and $0 < \lambda_1 \leq \dots \leq \lambda_n$. Thus, we have $\delta_t X_j = t^{\lambda_j} X_j$ for $1 \leq j \leq n$. Furthermore, each $X_j$ corresponds to a right-invariant vector field $X_j^R$ on $G$ which, in particular, satisfies
\begin{equation}\label{eq: right-invariant}
    \frac{\partial}{\partial t} g(\exp(tX_j) \bdot x) = (X_j^Rg)(\exp(tX_j) \bdot x)
\end{equation}
for all $g \in C^1(G)$ and $x \in G$, $t \in \R$. Finally, the map
\begin{equation*}
    \Upsilon \colon \R^n \to G, \qquad \Upsilon \colon (t_1, \dots, t_n) \mapsto \exp(t_nX_n) \bdot \dots \bdot \exp(t_1X_1)
\end{equation*}
is a global diffeomorphism (for a proof of this fact, see \cite[Lemma 1.31]{FS_book}). 

\subsection{Littlewood--Paley theory} Henceforth, we fix a homogeneous group $G$ with family of dilations $(\delta_t)_{t > 0}$. We shall use a variant of the classical Littlewood--Paley decomposition adapted to this setting. Consider a function $\psi \in C_c^{\infty}(G)$ which is \textit{mean zero} in the sense that
\begin{equation*}
    \int_G \psi = 0
\end{equation*}
and form the dilates
\begin{equation*}
    \psi_k(x) := 2^{-kQ}\psi \circ \delta_{2^{-k}}(x);
\end{equation*}
here $Q$ is the homogeneous dimension of $G$, as defined in \S\ref{subsec: basic concepts}. Consider the convolution operators $f \mapsto f \ast \psi_k$ for $f \in L^1(G)$. For a suitable choice of $\psi$, these operators play the role of classical Littlewood--Paley frequency projections.

\begin{proposition}\label{prop: Littlewood--Paley}
Let $G$ be a homogeneous group. There exists $\psi \in C_c^{\infty}(G)$ of mean zero such that
\begin{equation}\label{eq: fourier proj 2}
    f = \sum_{k \in \Z} f \ast \psi_k \qquad \textrm{for all $f \in  C_c(G)$,}
\end{equation}
where the convergence holds in the $L^p$ sense for all $1 < p \leq \infty$.
\end{proposition}

This result is well-known (for instance, it can be deduced from \cite[Proposition 3.4]{Christ_book}); however, for completeness, we present the straightforward proof in Appendix~\ref{appendix: Littlewood--Paley}.




\subsection{Standard reductions} 

We apply a sequence of standard reductions to reduce the study of the maximal operator to proving certain $L^2$ and weak-type $(1,1)$ bounds for a family \textit{linear} operators.




\subsubsection*{Cancellation} It is useful to introduce some additional cancellation in our operator. Let $\phi \in C_c^{\infty}(G)$ be a non-negative function such that
\begin{equation*}
    \int_{G} \phi = \sigma(G).
\end{equation*}
Note that $\nu := \sigma - \phi$ is a compactly supported, signed Borel measure on $G$ such that $\|\nu\| < \infty$ and $\nu$ is \textit{mean zero} in the sense that
\begin{equation*}
    \nu(G) = 0.
\end{equation*}
Clearly, we have a pointwise inequality
\begin{equation*}
    M[\sigma]f(x) \leq M[\nu]f(x) + M[\phi]f(x).
\end{equation*}
Moreover, $M[\phi]$ is a variant of the Hardy--Littlewood maximal function on $G$ which, as is well-known (see, for instance, \cite[Corollary 2.5]{FS_book}), is bounded on $L^p(G)$ for all $1 < p \leq \infty$. It therefore suffices to estimate the maximal function $M[\nu]$.




\subsubsection*{Domination by a square function} Using Proposition~\ref{prop: Littlewood--Paley}, we can choose a mean zero function $\psi \in C^{\infty}_c(G)$ which satisfies \eqref{eq: fourier proj 2}. Consider the family of maximal functions
\begin{equation*}
  M_{\ell}[\nu] f  :=  M[\psi_{\ell} \ast \nu]f = \sup_{k \in \Z} |f \ast \psi_{k+\ell} \ast \nu_k | \qquad \textrm{for $f \in C_c(G)$,}
\end{equation*}
defined for all $\ell \in \Z$, so that we may pointwise dominate
\begin{equation}\label{eq: sf domination}
 M[\nu] f(x) \leq \sum_{\ell \in \Z} M_{\ell}[\nu] f(x)  \leq \sum_{\ell \in \Z} S_{\ell}[\nu] f (x)
\end{equation}
where each $S_{\ell}[\nu]f$ is the associated square function
\begin{equation*}
    S_{\ell}[\nu]f := \Big(\sum_{k \in \Z} |f \ast \psi_{k+\ell} \ast \nu_k |^2 \Big)^{1/2}.
\end{equation*}
Theorem~\ref{thm: main} is therefore a consequence of the following proposition.

\begin{proposition}\label{prop: square function} For all $1 < p \leq 2$, there exists some $\varepsilon(p) >0$ such that 
\begin{equation*}
    \|S_{\ell}[\nu]f\|_{L^p(G)} \lesssim_{\sigma, p} 2^{-\varepsilon(p)|\ell|}\|f\|_{ L^p(G)} 
\end{equation*}
holds for all $f \in C_c(G)$ and all $\ell \in \Z$.
\end{proposition}

By our earlier observations (and, in particular, \eqref{eq: sf domination}), Proposition~\ref{prop: square function} immediately implies Theorem~\ref{thm: main} in the restricted range $1 < p \leq 2$; the remaining cases then follow via interpolation with the trivial bound for $p = \infty$.




\subsubsection*{Linearised square function} In order to prove Proposition~\ref{prop: square function}, we apply a standard randomisation procedure to linearise the square function. Let $r = (r_k)_{k \in \Z}$ be a sequence of IID random signs, with each $r_k$ taking the values $+1$ and $-1$ with equal probability. Consider the function 
\begin{equation*}
    T_{\ell,r}f := \sum_{k \in \Z} r_k T_{\ell}^kf \qquad \textrm{where} \qquad T_{\ell}^kf := f \ast \psi_{k+\ell} \ast \nu_k
\end{equation*}
In view of Khintchine's inequality, to prove Proposition~\ref{prop: square function}, it suffices to show that for all $1 < p \leq 2$ there exists some $\varepsilon(p) > 0$ such that the norm bound
  \begin{equation}\label{eq: linearised square function}
     \|T_{\ell,r}\|_{L^p(G) \to L^p(G)} \lesssim_p 2^{-\varepsilon(p)|\ell|}
 \end{equation} 
 holds for all $\ell \in \Z$ with a constant independent of realisation of $r$. Furthermore, \eqref{eq: linearised square function} in turn follows from a pair of endpoint estimates. 
 
\begin{lemma}\label{lem: almost orthogonal}
    There exists $\rho > 0$ such that
\begin{equation}\label{c1a.o.}
    \|(T_{\ell}^k)^{\ast}T_{\ell}^j\|_{L^2(G) \rightarrow L^2(G)} + \|T_{\ell}^k(T_{\ell}^j)^{\ast}\|_{L^2(G) \rightarrow L^2(G)} \lesssim_{\sigma} \min \{ 2^{ - \rho|\ell|}, 2^{-\rho|j-k|} \}
\end{equation}
 holds for all $j,k,\ell \in \Z$.
\end{lemma}

\begin{lemma}\label{lem: weak type 11} For all $\varepsilon > 0$, we have
\begin{equation}\label{c1weak11}
\|T_{\ell,r}\|_{L^{1}(G) \rightarrow L^{1,\infty}(G)} \lesssim_{\sigma, \varepsilon} 2^{\varepsilon|\ell|},
\end{equation}
where the implicit constant is independent of the realisation of $r$ and $\ell \in \Z$.
\end{lemma}

Once Lemma~\ref{lem: almost orthogonal} and Lemma~\ref{lem: weak type 11} are established, Proposition~\ref{prop: square function} (and hence Theorem~\ref{thm: main}) follow in a straightforward matter.

\begin{proof}[Proof of Proposition~\ref{prop: square function}]
Fix $\ell \in \Z$. By combining Lemma~\ref{lem: almost orthogonal} with the classical Cotlar--Stein lemma (see, for instance, \cite[Chapter VII, \S2]{Stein_book}), we deduce that there exists $\rho > 0$ (not necessarily the same as that in the statement of Lemma~\ref{lem: almost orthogonal}) such that 
\begin{align} \label{c1l2bound}
 \|T_{\ell,r}\|_{L^2(G) \rightarrow L^2(G)} \lesssim_{\sigma} 2^{-\rho|\ell|},
\end{align}
where the implicit constant is independent of the realisation of $r$. Fix $1 < p \leq 2$. By interpolating between \eqref{c1l2bound} and the weak $L^1$ estimate from Lemma~\ref{lem: weak type 11} (with $\varepsilon$ chosen sufficiently small, depending on $p$), we obtain \eqref{eq: linearised square function}. As previously noted, a standard argument using Khintchine's inequality then yields the square function bound in Proposition~\ref{prop: square function}.
\end{proof}




\subsection{Almost orthogonality} In this subsection we present the proof of Lemma~\ref{lem: almost orthogonal}. For this, it is convenient to introduce a slight generalisation of the curvature assumption. Let $\sigma$ be a finite Borel measure on $G$. We say that $\sigma$ satisfies $\textrm{{($\Sigma$CA)}}$ if it can be written as a sum $\sigma = \sigma_1 + \dots + \sigma_m$ where each $\sigma_j$ is a finite Borel measure on $G$ which satisfies $(\mathrm{CA})$. Lemma~\ref{lem: almost orthogonal} is then a fairly direct consequence of the following result. 

\begin{lemma}\label{lem: key L2 bound}
Suppose $\mu, \vartheta$ are finite, compactly supported Borel measures on $G$ which are mean zero and satisfy $(\Sigma\mathrm{CA})$. Then there exists $\rho > 0$ such that 
\begin{equation}\label{eq: key L2 bound}
  \|A[ \mu_a \ast \vartheta_b]\|_{L^2(G) \rightarrow L^2(G)} \lesssim_{\mu,\vartheta}
    2^{-\rho|a-b|}  \qquad \textrm{for all $a,b \in \Z$.}  
\end{equation}
\end{lemma}

\begin{proof}[Proof (Lemma~\ref{lem: key L2 bound} $\implies$ Lemma~\ref{lem: almost orthogonal})]
By unwinding definitions, we may write
\begin{align}
    (T_{\ell}^k)^{\ast}(T_{\ell}^j)  &= A[ \psi_{j+\ell} \ast \nu_j \ast  \tilde{\nu}_k \ast  \tilde{\psi}_{k+\ell}]   \label{c1eq: avgform_T1} 
 \end{align}
 and
 \begin{align}
      (T_{\ell}^k) (T_{\ell}^j)^{\ast}  &= A[\tilde{\nu}_j \ast \tilde{\psi}_{j+\ell}  \ast  \psi_{k+\ell} \ast \nu_k ]. \label{c1eq: avgform_T2}
\end{align}

Recall that the measures $\psi,\nu$ are mean zero. Furthermore, $\phi$ satisfies $(\mathrm{CA})$ and $\nu = \sigma - \phi$ is a difference of measures satisfying $(\mathrm{CA})$. We therefore have an estimate of the form \eqref{eq: key L2 bound} whenever the measures $\mu$ and $\vartheta$ are taken from the collection $\{\psi,\nu,\tilde{\psi},\tilde{\nu}\}$. In order to bound the operators in \eqref{c1eq: avgform_T1} and \eqref{c1eq: avgform_T2}, consider the pairs of measures listed in Table~\ref{tab: L2 bound}. Thus, by Lemma~\ref{lem: key L2 bound}, we can find $\rho > 0$ such that
\begin{align}
    \|A[{\psi}_{j+\ell} \ast \nu_j]\|_{L^2(G) \rightarrow L^2(G)} + \|A[  \tilde{\nu}_j \ast \tilde{\psi}_{j+\ell}]\|_{L^2(G) \rightarrow L^2(G)} & \lesssim_{\sigma} 2^{-\rho|\ell|}, \label{c1eq: avgest1}
\end{align}
and 
\begin{align}
    \|A[  \nu_j \ast \tilde{\nu}_k]\|_{L^2(G) \rightarrow L^2(G)} + \|A[\tilde{\psi}_{j+\ell} \ast {\psi}_{k+\ell}]\|_{L^2(G) \rightarrow L^2(G)} & \lesssim_{\sigma} 2^{-\rho|k - j|}. \label{c1eq: avgest2} 
\end{align}

\begin{table}
    \centering
    \renewcommand{\arraystretch}{1.5} 
\begin{tabular}{ | m{7em} | m{1em} |  m{2em}| m{2em} | m{2em} | m{7em} |} 
  \hline
  Operator & $\mu$ & $\vartheta$ & $a$ & $b$ & $2^{-|a-b|} $ \\
  \hline 
  \hline 
 $A[ {\psi}_{j+\ell} \ast \nu_j ]$ & $\psi$ & $\nu$ & $j+\ell$ & $j$& $2^{-|\ell|}$ \\
  \hline
  $A[ \nu_j \ast \tilde{\nu}_k]$ & $\nu$ & $\tilde{\nu}$ & $j$ & $k$ & $2^{-|k - j|}$ \\ 
 \hline 
  $A[ \tilde{\nu}_j \ast \tilde{\psi}_{j+\ell}]$ & $\tilde{\nu}$ & $\tilde{\psi}$ & $j$ & $j+\ell$ & $2^{-|\ell|}$ \\ 
  \hline
  $A[ \tilde{\psi}_{j+\ell} \ast {\psi}_{k+\ell}]$ & $\tilde{\psi}$ & $\psi$ & $j+\ell$ & $k+\ell$& $2^{-|k-j|}$ \\
  \hline
\end{tabular}
\vspace{5pt}
    \caption{Applications of Lemma~\ref{lem: key L2 bound}.}
    \label{tab: L2 bound}
\end{table}

To deduce the statement of Lemma~\ref{lem: almost orthogonal} from the above estimates, we repeatedly apply Young's inequality\footnote{For Young's inequality over homogeneous groups, see \cite[Proposition 1.18]{FS_book}. The reference only treats the functional case; however, the version for measures follows in a similar manner.} to deduce that
\begin{align*}
    & \|(T_{\ell}^k)^{\ast} (T_{\ell}^j)\|_{L^2(G) \rightarrow L^2(G)} \\
    & \qquad \lesssim_{\sigma} \min \left \{ \|A[{\psi}_{j+\ell} \ast \nu_j]\|_{L^2(G) \rightarrow L^2(G)}, \|A[  \nu_j \ast \tilde{\nu}_k]\|_{L^2(G) \rightarrow L^2(G)} \right \}
\end{align*} 
and 
\begin{align*}
     &\|(T_{\ell}^k) (T_{\ell}^j)^{\ast}\|_{L^2(G) \rightarrow L^2(G)}\\
    & \qquad  \lesssim_{\sigma}  \min \left\{ \|A[  \tilde{\nu}_j \ast \tilde{\psi}_{j+\ell}]\|_{L^2(G) \rightarrow L^2(G)}, 
\|A[\tilde{\psi}_{j+\ell} \ast {\psi}_{k+\ell}]\|_{L^2(G) \rightarrow L^2(G)} \right\}.
\end{align*}
Combining these inequalities with \eqref{c1eq: avgest1} and \eqref{c1eq: avgest2}, we obtain \eqref{c1a.o.}, which completes the proof of Lemma~\ref{lem: almost orthogonal}.
\end{proof}

\begin{proof}[Proof (of Lemma~\ref{lem: key L2 bound})] Fix $\mu$ and $\vartheta$ as in the statement of the lemma and $a$, $b \in \Z$ with $a \leq b$. The adjoint of $A[\mu_a \ast \vartheta_b]$ is the operator $A[\tilde{\vartheta}_b \ast \tilde{\mu}_a]$. Since the hypotheses on $\mu$ and $\vartheta$ are preserved under reflection, it suffices to show there exists some $\rho > 0$ such that
    \begin{align*}
        \|A[\mu_a \ast \vartheta_b]\|_{L^2(G) \rightarrow L^2(G)} \lesssim_{\mu, \vartheta}
    2^{\rho(a-b)}.
    \end{align*}
By linearity, we may further assume without loss of generality that $\vartheta$ satisfies $(\mathrm{CA})$. 

By a simple scaling argument, 
\begin{equation} \label{c1eq: simplerescl}
    \|A[\mu_a \ast \vartheta_b]\|_{L^2(G) \rightarrow L^2(G)} = \|A[\mu_{a-b} \ast \vartheta]\|_{L^2(G) \rightarrow L^2(G)}.
\end{equation}
Let $\vartheta^{(n)}$ denote the $n$-fold convolution product as introduced in \eqref{eq: conv prod} and, for notational convenience, write
\begin{equation*}
    A^{(n)} := A[\mu_{a-b} \ast \vartheta^{(n)}].
\end{equation*}
We claim that 
\begin{align}\label{c1eq: inducgoal}
    \|A^{(n)}\|_{L^2(G) \rightarrow L^2(G)} \leq \|\mu\|^{1/2}\|\vartheta\|^{n/2}\|A^{(n+1)}\|_{L^2(G) \rightarrow L^2(G)}^{1/2}
\end{align}
holds for any $n \in \N_0$.

To prove the claim, we use the Hilbert space identity 
\begin{equation*}
    \|A^{(n)}\|_{L^2(G) \rightarrow L^2(G)} = \|(A^{(n)})^{\ast} ( A^{(n)})\|_{L^2(G) \rightarrow L^2(G)}^{1/2}.
\end{equation*}
However,
\begin{equation*}
    (A^{(n)})^{\ast} ( A^{(n)}) = A[\mu_{a-b} \ast \vartheta^{(n)} \ast \mathcal{R
}{({\vartheta}^{(n)})} \ast \widetilde{\mu}_{a-b} ]
\end{equation*}
 where $\mathcal{R}$ maps a measure $\varrho$ to its reflection $\tilde{\varrho}$. In view of \eqref{eq: conv prod}, we deduce that
 \begin{equation*}
\mu_{a-b} \ast \vartheta^{(n)}\ast \mathcal{R}(\vartheta^{(n)}) \ast \tilde{\mu}_{a-b} = 
\mu_{a-b} \ast \vartheta^{(n+1)}\ast \mathcal{R}({\vartheta}^{(n-1)}) \ast \tilde{\mu}_{a-b} .
 \end{equation*}
 
  Thus, by repeated applications of Young's inequality, we obtain the claim \eqref{c1eq: inducgoal} (note that the $(n-1)$th convolution product, as defined by \eqref{eq: conv prod}, involves the convolution of $n$ measures). 
  
  By \eqref{c1eq: simplerescl} and repeated application of \eqref{c1eq: inducgoal}, we deduce that 
 \begin{equation}\label{c1eq: cumgoal}
   \|A[\mu_a \ast \vartheta_b]\|_{L^2(G) \rightarrow L^2(G)} \lesssim_{\mu, \vartheta} \|A^{(n)}\|_{L^2(G) \rightarrow L^2(G)}^{1/2^{n}} \quad \text{for all $n \in \N_{0}$.}
 \end{equation}
Since $\vartheta$ satisfies $(\mathrm{CA})$, there exists some $N \in \N_{0}$ such that $\vartheta^{(N)}$ is absolutely continuous with respect to the Haar measure on $G$ with Radon--Nikodym derivative $h$ and, furthermore, there exists some $\gamma > 0$ such that
\begin{equation}\label{eq: CA vartheta}
    \int_{G}\left|h(y^{-1} \bdot x) - h(x)\right| \,\ud x  \lesssim_{\vartheta} |y|_G^{\gamma} \qquad \textrm{for all $y \in G$.}
\end{equation}
In view of \eqref{c1eq: cumgoal}, it suffices to estimate the operator norm of $A^{(N)}$ for this choice of exponent $N$. 

By Young's inequality, the operator bound for $A^{(N)}$ follows from an estimate on $\|\mu_{a-b} \ast \vartheta^{(N)}\|_{L^1(G)}$. Since $\mu$ has mean zero, it follows that
\begin{equation*}
\| \mu_{a-b} \ast \vartheta^{(N)}\|_{L^1(G)} = \|\mu_{a-b} \ast h\|_{L^1(G)} = \int_{G}\left|\int_{G}\left(h(y^{-1} \bdot x) - h(x)\right) \,\ud \mu_{a-b}(y) \right|\,\ud x .
\end{equation*}
Applying the triangle inequality, the Fubini--Tonelli theorem and \eqref{eq: CA vartheta}, we have
\begin{equation*}
\|\mu \ast \vartheta^{(N)}\|_{L^1(G)}
 \lesssim_{\vartheta} \int_{G}|y|_G^{\gamma} \,\ud |\mu_{a-b}|(y) = 2^{\gamma(a-b)}  \int_{G}|y|_G^{\gamma} \,\ud |\mu|(y).
\end{equation*}
Consequently,
 \begin{align*}
    \|A[\mu_a \ast \vartheta_b]\|_{L^2(G) \rightarrow L^2(G)} \lesssim_{\mu, \vartheta} 2^{\rho(a-b)} \qquad \textrm{for $\rho := 2^{-N} \gamma$}
\end{align*} 
 and this concludes the proof.
\end{proof}

\subsection{Calder\'{o}n--Zygmund estimates} In this subsection we present the proof of Lemma~\ref{lem: weak type 11}. The argument is based on Calder\'{o}n--Zygmund theory adapted to homogeneous groups.

\begin{proof}[Proof (of Lemma~\ref{lem: weak type 11})] Let $K_{\ell}$ and $K_{\ell}^k$ denote the kernels of $T_{\ell,r}$ and $T_{\ell}^k$, respectively, so that 
\begin{equation}\label{c1eq: kern_Kl}
    K_{\ell} = \sum_{k \in \Z} r_k K_{\ell}^k \qquad \textrm{and} \qquad K_{\ell}^k =  \psi_{k + \ell} \ast \nu_k = ( \psi_{\ell} \ast \nu)_k.
\end{equation}
By Calder\'{o}n--Zygmund theory adapted to the homogeneous  group setting (see, for instance, \cite[Chapter I, \S5]{Stein_book}), to prove \eqref{c1weak11} it suffices to verify the H\"{o}rmander condition
\begin{equation}\label{c1hormandertype}
    \sup_{y \in G} \int_{|x|_G \geq C_0|y|_G}|K_{\ell}(y^{-1} \bdot x) - K_{\ell}(x)|\,\ud x \lesssim_{\varepsilon}2^{\varepsilon|\ell|} \qquad \textrm{for all $\varepsilon > 0$}
\end{equation}
for some fixed constant $C_0 > 1$. In fact, we shall take $C_0 := 2C_G$ for $C_G \geq 1$ the constant appearing in the quasi triangle inequality \eqref{eq: quasi triangle}. 

In view of \eqref{c1eq: kern_Kl}, it is clear that \eqref{c1hormandertype} follows from the estimate 
\begin{equation}\label{c1sumikl}
    \sup_{y \in G} \ \sum_{k \in \Z} I_{\ell}^k(y) \lesssim_{\varepsilon}2^{\varepsilon|\ell|} \qquad \textrm{for all $\varepsilon > 0$}
\end{equation}
where
\begin{equation*}
    I_{\ell}^k(y) := \int_{|x|_G \geq C_0|y|_G}|K_{\ell}^k(y^{-1} \bdot x) - K_{\ell}^k(x)|\,\ud x.
\end{equation*}

To prove \eqref{c1sumikl}, we first identify $y \in G$ for which $I_{\ell}^k(y) = 0$. By unwinding the definition of $K_{\ell}^{k}$, we may write
\begin{equation}\label{c1Ikl_step2}
    I_{\ell}^k(y) = \int_{|x|_G \geq C_02^{-k}|y|_G}|(\psi_{\ell} \ast \nu)((\delta_{2^{-k}} y)^{-1}\bdot x) - (\psi_{\ell} \ast \nu)(x)|\,\ud x. 
\end{equation}
As $\nu$ and $\psi$ are compactly supported, we can find some $R > 1$ such that the support of $\psi_{\ell} \ast \nu$ is contained inside the ball $B(0, R 2^{\max\{\ell,0\}})$. Set $C := 2R$. We claim that
\begin{align}\label{c1eq: Ikl vanish}
    I_{\ell}^k(y) = 0 \qquad \text{whenever} \quad |y|_G \geq C 2^{ k + \max\{\ell,0\}}.
\end{align}
To see this, fix $y \in G$ such that $|y|_G \geq C 2^{ k + \max\{\ell,0\}}$. In view of \eqref{c1Ikl_step2}, it suffices to check that 
\begin{equation*}
    |(\delta_{2^{-k}}y)^{-1}\bdot x|_G,\; |x|_G \geq C2^{\max\{\ell, 0\}} \qquad \text{whenever} \quad |x|_G \geq C_{0}|\delta_{2^{-k}}y|_G.
\end{equation*}
The lower bound on $|x|_G$ is immediate. On the other hand, by \eqref{eq: quasi triangle} and our choice of $C_0$, we deduce that
\begin{equation*}
  |(\delta_{2^{-k}}y)^{-1}\bdot x|_G \geq  |\delta_{2^{-k}}y|_G \geq C2^{\max\{\ell, 0\}}.
\end{equation*} 
Thus, we have established \eqref{c1eq: Ikl vanish}.

By Young's inequality, the $L^1$ norm of the kernel $K_{\ell}^k$ is uniformly bounded. Consequently, we have the uniform estimate
 \begin{equation}\label{eq: CZ uniform bound}
     I_{\ell}^k(y) \lesssim 1 \qquad \text{ for any $k,\ell \in \Z$.}
 \end{equation}
In order to sum in $k$, we shall improve over \eqref{eq: CZ uniform bound} for certain exponent pairs by establishing an estimate with geometric decay in $k + \ell$. The key tool is a variant of the mean value theorem on $G$.   

  \begin{lemma}[Mean value theorem]\label{lem: mean value} Let $g \in C_c^{1}(G)$. For any $z \in G$, we have
\begin{equation*}
 \int_{G}|g(z \bdot x) - g(x)|\,\ud x \lesssim \sum_{j = 1}^n |z|_G^{{\lambda_j}}\|X^R_{j}g\|_{L^1(G)}. 
\end{equation*}
\end{lemma}

Lemma~\ref{lem: mean value} is a variant of \cite[Theorem 1.33]{FS_book}; for completeness we present the proof at the end of the section. Here the $\lambda_j$ are the eigenvalues of the dilation matrix $\Lambda$ and the $X_j^R$ are right-invariant vector fields as defined in \S\ref{subsec: basic concepts}.

Fix $k,\ell \in \Z$. After relaxing the region of integration in \eqref{c1Ikl_step2}, we see that 
\begin{align*}
I_{\ell}^k(y)  &\leq    \int_{G} |(\psi_{\ell} \ast \nu)((\delta_{2^{-k}} y)^{-1}\bdot x) - (\psi_{\ell} \ast \nu)(x)|\,\ud x  \\
& = \int_{G} \left |\int_{G} \psi_{\ell}((\delta_{2^{-k}} y)^{-1}\bdot x\bdot z^{-1}) - \psi_{\ell}(x\bdot z^{-1}) \,\ud \nu(z) \right | \,\ud x \\
& \leq \|\nu\|\int_{G} |\psi_{\ell}((\delta_{2^{-k}} y)^{-1}\bdot x) - \psi_{\ell}(x)| \,\ud x. 
\end{align*}
We apply Lemma~\ref{lem: mean value} with $g = \psi_{\ell}$ and $z = (\delta_{2^{-k}}y)^{-1}$. As a consequence of \eqref{eq: right-invariant} and the identity $\delta_{2^{-\ell}}X_j = 2^{-\ell \lambda_j} X_j$, we have $X_j^R\psi_{\ell} = 2^{-\ell\lambda_j}(X^R_{j}\psi)_{\ell}$. We therefore deduce that
 \begin{equation*}
I_{\ell}^k(y)
     \lesssim_{\sigma} \sum_{j = 1}^n (2^{-k}|y|_G)^{{\lambda_j}}\|X^R_{j}(\psi_{\ell})\|_{L^1(G)} = \sum_{j = 1}^n (2^{-(k+\ell)}|y|_G)^{{\lambda_j}}\|X^R_{j}\psi\|_{L^1(G)}.
\end{equation*}
Combining this with the trivial estimate \eqref{eq: CZ uniform bound}, we have
\begin{equation}\label{c1eq: non_triv_est}
    I_{\ell}^k(y)
     \lesssim_{\sigma}  \min \Big\{ 1, \sum_{j = 1}^n (2^{-(k+\ell)}|y|_G)^{\lambda_j} \Big\}.
\end{equation}

To conclude the proof, we consider two separate cases.

\medskip

\noindent \textit{Case 1: $\ell \geq 0$.} By \eqref{c1eq: Ikl vanish} and \eqref{c1eq: non_triv_est}, we deduce that 
\begin{align*}
\sum_{ k \in \Z} I_{\ell}^k(y) & = \sum_{\substack{k \in \Z \\ |y|_G \leq C2^{\ell + k} }} I_{\ell}^k(y) 
& \lesssim_{\sigma} \sum_{j = 1}^{n} (2^{-\ell}|y|_G)^{\lambda_j}\Big(\sum_{\substack{k \in \Z \\ C^{-1}2^{-\ell}|y|_G \leq 2^k}} 2^{-k\lambda_j} \Big) 
   & \lesssim 1.
\end{align*}
This is a strong version of the desired inequality \eqref{c1sumikl}.

\smallskip

\noindent \textit{Case 2: $\ell < 0$.} In view of \eqref{c1eq: Ikl vanish}, we may split the sum
\begin{equation*}
\sum_{k \in \Z} I_{\ell}^k(y) = 
\sum_{\substack{k \in \Z \\ |y|_G \leq C 2^k}} I_{\ell}^k(y) =  \sum_{k \in \mathcal{A}_{\ell}(y)} I_{\ell}^k(y) +  \sum_{k \in \mathcal{B}_{\ell}(y)} I_{\ell}^k(y)
\end{equation*}
where
\begin{align*}
    \mathcal{A}_{\ell}(y) &:= \{k \in \Z : 2^{-k}|y|_G \leq C2^{\ell}\}, \\
    \mathcal{B}_{\ell}(y) &:= \{k \in \Z : C2^{\ell} < 2^{-k}|y|_G \leq C\}.
\end{align*}
We apply \eqref{c1eq: non_triv_est} to each of the terms in the right-hand sums. For large values of $k$, this leads to the estimate  
\begin{equation*}
    \sum_{k \in \mathcal{A}_{\ell}(y)} I_{\ell}^k(y) \lesssim_{\sigma} \sum_{j = 1}^{n} (2^{-\ell}|y|_G)^{\lambda_j}\Big(\sum_{\substack{k \in \Z \\ \ C^{-1}2^{-\ell}|y|_G \leq 2^k}} 2^{-k\lambda_j} \Big) \lesssim 1.
\end{equation*}
On the other hand, for the small values of $k$, we have
\begin{equation*}
    \sum_{k \in \mathcal{B}_{\ell}(y)} I_{\ell}^k(y) \lesssim_{\sigma} \#\mathcal{B}_{\ell}(y) \lesssim |\ell|. 
\end{equation*}
Thus, we again have the desired inequality \eqref{c1sumikl}.\smallskip

In either case, we obtain \eqref{c1sumikl}, which completes the proof of Lemma~\ref{lem: weak type 11}.
\end{proof}




\subsection{A mean value estimate} It remains to provide the details of the proof of Lemma~\ref{lem: mean value}.

\begin{proof}[Proof (of Lemma~\ref{lem: mean value})] First consider an element of the form $z_j = \exp(\pm t_jX_j)$ for some $1 \leq j \leq n$ and $t_j > 0$. By the right-invariance property \eqref{eq: right-invariant}, we have
\begin{equation*}
 |g(z_j\bdot x) - g(x)|   = \Big|\int_0^{t_j} \frac{\partial}{\partial s} g(\exp( \pm s X_j) \bdot  x) \,\ud s\Big|  = \Big|\int_0^{t_j} (X_j^Rg)(\exp(\pm s X_j) \bdot x) \,\ud s\Big|.
\end{equation*}
Therefore, by changing the order of integration and using the translation invariance of the Haar measure,
\begin{equation}\label{eq: MV 1}
    \int_G |g(z_j\bdot x) - g(x)|\,\ud x  \leq  |t_j| \|X_j^R g\|_{L^1(G)}. 
\end{equation}

Now consider an arbitrary element $z \in G \setminus \{0\}$. Since, as discussed in \S\ref{subsec: basic concepts}, the map
\begin{equation*}
\Upsilon(t_1, \dots, t_n) := \exp(t_nX_n) \bdot \dots \bdot \exp(t_1X_1)
\end{equation*}
is a global diffeomorphism, $z$ can be written uniquely as
\begin{equation*}
    z = z_n \bdot \dots \bdot z_1 \qquad \textrm{where} \qquad z_j := \exp(t_jX_j) 
\end{equation*} 
for some $t_1,\dots, t_n \in \R$. Define a sequence of group elements by taking $\zeta_0 := 0$ and
\begin{equation*}
    \zeta_j := z_j \bdot \dots \bdot z_1 \qquad \text{ for $1 \leq j \leq n$,}
\end{equation*}
so that $z = \zeta_n$. We therefore have
\begin{equation*}
    g(z \bdot x) - g(x) = \sum_{j = 1}^n \big(g(\zeta_j \bdot x) - g(\zeta_{j-1} \bdot x)\big) = \sum_{j = 1}^n \big(g(z_j \bdot x_j) - g(x_j)\big)
\end{equation*} 
where $x_j := \zeta_{j-1} \bdot x$. By repeated application of the inequality \eqref{eq: MV 1} from the case considered above, 
\begin{equation}\label{eq: MV 2}
    \int_{G}|g(z \cdot x) - g(x)|\ud x \leq \sum_{j = 1}^n  |t_j| \|X_j^R g\|_{L^1(G)} \lesssim \sum_{j = 1}^n  |z|_G^{\lambda_j} \|X_j^R g\|_{L^1(G)},
\end{equation}
as required. Note that the final step in \eqref{eq: MV 2} follows from the inequality
\begin{equation*}
    |z|_G^{-1}\sum_{j=1}^n |t_j|^{1/\lambda_j} \leq \sup \Big\{ \sum_{j = 1}^{n} |s_j|^{1/\lambda_j} : |\exp(s_nX_n)\bdot \dots \bdot \exp(s_1X_1)|_G \leq 1 \Big\} \lesssim 1,
\end{equation*}
which in turn is a simple consequence of the definition of the group dilations. 
\end{proof}




\section{Analysis of the examples}\label{sec: analysis examples}




\subsection{Graded and stratified groups}\label{subsec: graded and stratified groups} We investigate example applications of Theorem~\ref{thm: main} in the setting of graded and stratified groups; the relevant definitions are recalled presently (see \cite[Chapter 1]{FS_book} for further details). 

A Lie algebra $\fg$ with Lie bracket $[\,\cdot\,,\,\cdot\,]_{\fg}$ is \textit{graded} if there exists a vector space decomposition
\begin{equation}\label{eq: graded}
    \fg = \bigoplus_{j=1}^{\infty} V_j \qquad \textrm{where} \qquad [V_i, V_j]_{\fg} \subseteq V_{i+j} \quad \textrm{for $i, j \geq 1$;}
\end{equation}
here all but finitely many of the vector spaces $V_j$ are equal to $\{0\}$. If $G$ is a connected, simply connected nilpotent Lie group with graded Lie algebra $\fg$, then there is automatically a natural dilation structure on $\fg$ induced by the grading and so $G$ is a homogeneous group.

A Lie algebra $\fg$ is \textit{stratified} if it is graded and $V_1$ generates $\fg$ as an algebra; in this case, if $\fg$ is nilpotent of step $m$, then
\begin{equation}\label{eq: stratified}
    \fg = \bigoplus_{j=1}^m V_j \qquad \textrm{where} \qquad [V_1, V_j]_{\fg} = V_{j+1} \quad \textrm{for $1 \leq j \leq m$.}
\end{equation}
We say a Lie group $G$ is a \textit{graded} (respectively, \textit{stratified}) \textit{group} if it a homogeneous group such that the Lie algebra $\fg$ is graded (respectively, stratified).

We can relate the Lie bracket on $\fg$ to the group theoretic commutator on $G$ via the Baker--Campbell--Hausdorff formula. In particular, for $x$, $y \in G$ define 
\begin{equation*}
    [x, y]_G := x\bdot y \bdot x^{-1} \bdot y^{-1}. 
\end{equation*}
Then, by the Baker--Campbell--Hausdorff formula,
\begin{equation}\label{eq: BCH}
    [\exp(X), \exp(Y)]_G = \exp\big([X, Y]_{\fg} +  e_3(X,Y)\big)
\end{equation}
for all $X,Y \in \fg$, where $e_3(X,Y)$ is a linear combination of lie brackets of $X$ and $Y$ of order at least 3.




\subsection{Testing conditions for analytic submanifolds}

    Let $G$ be a homogeneous group and $S$ be a smooth submanifold $G$. We say a Borel measure $\sigma$ on $G$ is a \textit{$C_c^{\infty}$-density on $S$} if it is of the form $\eta \ud \sigma_S$ where $\eta \in C^{\infty}_c(S)$ is a smooth, compactly supported function on $S$ and $\sigma_S$ is the natural surface measure on $S$ induced by the Haar measure on $G$.

We consider testing conditions which ensures any $C^{\infty}_c$-density on $S$ satisfies $(\mathrm{CA})$. For analytic submanifolds, a very simple sufficient condition is provided by Ricci--Stein~\cite{RS1988}. 

\begin{proposition}[Corollary 2.3, \cite{RS1988}]\label{prop: RS analytic} Let $S$ be a connected analytic submanifold of a homogeneous group $G$. If $S$ generates the group $G$, then any $C^{\infty}_c$-density $\sigma$ on $S$ satisfies $(\mathrm{CA})$. 
\end{proposition}

Here we say a set \textit{$S \subseteq G$ generates $G$} if $G = \langle S \rangle$ where 
\begin{equation}\label{eq: subgroup gen}
    \langle S \rangle := \{ s_1 \bdot \dots \bdot s_N : s_1, \dots, s_N \in S \cup \tilde{S} \}
\end{equation}
for $\tilde{S} := \{s^{-1} : s \in S\}$. Ricci--Stein~\cite{RS1988} work with the ostensibly weaker condition that $G = \mathrm{clos}(\langle S \rangle)$; however, in all cases we consider (that is, for $S$ a connected analytic submanifold) these conditions turn out to be equivalent.

We remark that the result in \cite[Corollary 2.3]{RS1988} is in fact somewhat more general. There, the authors consider a family of connected analytic submanifolds $S_j$ for $1 \leq j \leq N$ such that the iterated product set $S_1 \bdot \dots \bdot S_N$ contains a non-trivial open subset of $G$.  For each $j$, one fixes $\sigma_j$ a smooth density on $S_j$ and considers the convolution product $\sigma_1 \ast \cdots \ast \sigma_N$. To recover Proposition~\ref{prop: RS analytic}, we choose the $S_j$ to alternate between $S$ and the reflection $\tilde{S}$ and, accordingly, the  $\sigma_j$ to alternate between $\sigma$ and $\tilde{\sigma}$. Using \cite[Proposition 1.1]{RS1988}, the hypothesis that $S$ generates $G$ implies the existence of some $N$ such that $S_1 \bdot \dots \bdot S_N$ contains a non-trivial open subset of $G$, and so \cite[Corollary 2.3]{RS1988} applies.\footnote{Alternatively, if $e \in S$ (which we may always assume in applications to maximal functions), then the condition that $S$ generates $G$ is equivalent to 
\begin{equation*}
    G = \bigcup_{N=1}^{\infty} S_1 \bdot \dots \bdot S_N
\end{equation*}
for $S_j$ as defined above. It is then easy to show there exists some $N \in \N$ such that $S_1 \bdot \dots \bdot S_N$ contains a non-trivial open set using the Baire category theorem.}




\subsection{The Kor\'anyi sphere and extensions}\label{subsec: Koranyi}

We return to the example of the Kor\'anyi sphere and its extensions, as discussed in \S\ref{subsec: examples}, 2). Here we work in the setting of a graded Lie group $G$ with $\dim G \geq 2$. Using Proposition~\ref{prop: RS analytic}, we verify $(\mathrm{CA})$ for a large class of measures which, in exponential coordinates, are supported on boundaries of convex domains. 

\begin{lemma}\label{lem: convex ex} Let $G$ be a graded Lie group with $\dim G \geq 2$ and suppose $\Omega$ is an open convex domain in $\fg$ with analytic boundary $\Sigma := \partial \Omega$. Then any $C_c^{\infty}$-density on $\exp(\Sigma)$ satisfies $(\mathrm{CA})$. 
\end{lemma}

Lemma~\ref{lem: convex ex} applies to the Kor\'anyi sphere in the Heisenberg group and therefore, in view of Theorem~\ref{thm: main}, we obtain a significant extension of~\cite[Theorem 1.2]{GT2021}.

\begin{proof}[Proof (of Lemma~\ref{lem: convex ex})] By Proposition~\ref{prop: RS analytic}, it suffices to show that $\exp(\Sigma) \bdot \exp(\Sigma)$ contains an open ball in $G$ (defined with respect to, say, the homogeneous norm). Indeed, if this is the case, then $\exp(\Sigma)$ must generate an open ball $B$ around the origin. Using the Baker--Campbell--Hausdorff formula, given $x \in G$, there exists some $y \in B$ and $N \in \N$ such that $x = y^N = y \bdot \dots \bdot y$, and so $\exp(\Sigma)$ generates $G$. 

Let $\fg$ be the Lie algebra associated to $G$, which we assume admits a grading as in \eqref{eq: graded}. For $X \in \fg$, define the linear map 
\begin{equation*}
    \Phi_{X} \colon \fg \to \fg, \qquad \Phi_X(Y) \colon Y \mapsto [X,Y]_{\fg}.
\end{equation*}   
The key claim is that for any $X \in \fg$, the kernel $\ker \Phi_X$ has dimension at least 2.

Temporarily assuming the claim, we argue as follows. Assume, without loss of generality, that $B_G(0,1) : = \{x \in G : |x|_G < 1\}\subseteq \exp(\Omega)$. Choose $x \in B_G(0,1)$ and let $X := \exp^{-1}(x)$. Let $H_X$ be a subspace of $\ker \Phi_X$ of dimension $2$ and let $S^{n-1}$ denotes the unit sphere in $\fg$ with respect to the euclidean norm. By the convexity of $\Omega$, for each $W \in S^{n-1} \cap H_X$ there exist unique real numbers $t(W)$, $s(W) >0$ such that 
\begin{equation*}
    X + t(W)W, \quad X - s(W)W \in \Sigma.
\end{equation*}
where $t(W)W$ and $s(W)W$ are the usual scalar multiples of $W$. Furthermore, the mapping
\begin{equation*}
 F \colon S^{n-1} \cap H_X \rightarrow \R, \qquad  F \colon W \mapsto t(W) - s(W)
\end{equation*}
is continuous. Clearly, $t(-W) = s(W)$ and $s(-W) = t(W)$, so that $F(-W) = -F(W)$. The set $S^{n-1} \cap H_X$ is a $1$-dimensional (euclidean) sphere. By the intermediate value theorem, there exists some $W_x \in S^{n-1} \cap H_X$ such that $F(W_x) = 0$ or, equivalently, $t(W_x) = s(W_x)$. Since $W_x \in  H_X$, by the Baker--Campbell--Hausdorff formula and bilinearity of the Lie bracket, 
\begin{equation*}
    x \bdot x = x \bdot w_x \bdot w_x^{-1} \bdot x \in \exp(\Sigma) \bdot \exp(\Sigma), \qquad \textrm{where $w_x := \exp(t(W_x)W_x)$.}
\end{equation*}
As $x \in B_G(0,1)$ was chosen arbitrarily, we conclude that $\exp(\Sigma) \bdot \exp(\Sigma)$ contains an open ball in $G$.

It remains to verify the claim that $\dim \ker \Phi_X \geq 2$ for all $X \in \fg$. From the definition of the grading, the image of $\Phi_X$ is always contained in $\bigoplus_{j \ge 2} V_j$. Thus, if $\dim(V_1) \geq 2$, then the result immediately follows from the rank--nullity theorem. On the other hand, if $\dim(V_1) = 1$, then we must have $[V_1, V_1] = \{0 \}$. Therefore, the image $\Phi_X$ is contained in  $\bigoplus_{j \ge 3} V_j$ and we can again apply rank-nullity to deduce the desired result.
\end{proof}




\subsection{Horizontal spheres and extensions}\label{subsec: horizontal} We return to the example of the horizontal spheres and the various extensions discussed in \S\ref{subsec: examples}, 3) and 4). Here it is natural to work with a stratified Lie group $G$, so that the Lie algebra $\fg$ admits a stratification as in \eqref{eq: stratified}. Let $\Pi_1 \colon \fg \to V_1$ denote the subspace projection onto $V_1$. The main result is as follows.

\begin{lemma}\label{lem: projection example}
Let $G$ be a stratified Lie group with $\dim G \geq 2$ and $\Sigma \subseteq \fg$ be a connected analytic submanifold such that $\Pi_1(\Sigma)$ generates $V_1$ (in terms of vector addition). Any $C_c^{\infty}$-density on $\exp(\Sigma)$ satisfies $(\mathrm{CA})$.
\end{lemma}

Clearly Lemma~\ref{lem: projection example} applies to the horizontal and tilted sphere examples in Heisenberg (and M\'etivier) groups and therefore, combined with Theorem~\ref{thm: main}, we obtain a significant extension of 
\cite[Theorem 1.1]{BHRT2021} (and also $L^p$ boundedness results mentioned in passing in \cite{RSS2022}). Of course, Lemma~\ref{lem: projection example} has a much broader scope, and provides a rich class of examples of arbitrary dimension which need not be associated to any $d$-plane distribution in the group. 

In view of Proposition~\ref{prop: RS analytic}, the proof of Lemma~\ref{lem: projection example} is reduced to showing the following lemma. 

\begin{lemma}[Generator test]\label{lem: gen test} Let $G$ be a stratified Lie group and $S \subseteq G$. Then $S$ generates $G$ if and only if $\Pi_1\circ \exp^{-1}(S)$ generates $V_1$ (in terms of vector addition).
\end{lemma}

Before presenting the proof, we introduce some helpful notation and consequences of the Baker--Campbell--Hausdorff formula \eqref{eq: BCH}.

Let $G$ be a graded Lie group with Lie algebra $\fg$. For $1 \leq \ell \leq m$, define the mapping 
\begin{equation*}
  \Phi_{\ell} \colon \fg^{\ell} \to \fg, \qquad \Phi_{\ell} \colon (X_1, \dots, X_{\ell}) \mapsto [X_1, [X_2, \dots , [X_{\ell-1}, X_{\ell}]_{\fg} \dots ]_{\fg}]_{\fg},
\end{equation*}
which takes a nested sequence of Lie brackets of $\ell$ algebra elements. Note that $\Phi_{\ell}$ maps the subspace $V_1^{\ell}$ into $V_{\ell}$. If $G$ is stratified, then this restricted mapping is a surjection. 

On the other hand, for $1 \leq \ell \leq m$, define the mapping 
\begin{equation*}
  \phi_{\ell} \colon G^{\ell} \to G, \qquad \phi_{\ell} \colon (x_1, \dots, x_{\ell}) \mapsto [x_1, [x_2, \dots , [x_{\ell-1}, x_{\ell}]_G \dots ]_G]_G,
\end{equation*}
which takes a nested sequence of commutators of $\ell$ group elements.
By iteratively applying the Baker--Campbell--Hausdorff formula \eqref{eq: BCH}, we have
\begin{equation}\label{eq: l-fold BCH}
\phi_{\ell}(\exp(X_1), \dots, \exp(X_{\ell})) = \exp\big(\Phi_{\ell}(X_1, \dots, X_{\ell}) + e_{\ell + 1}(X_1, \cdots, X_{\ell})\big),
\end{equation}
for all $X_1, \dots, X_{\ell} \in \fg$, where $e_{\ell + 1}(X_1, \cdots, X_{\ell})$ is a linear combination of Lie brackets of $X_1, \dots, X_{\ell}$ of order at least $\ell + 1$. 

Generalising the definition of $\Pi_1$ introduced above, for $j \in \N$, let $\Pi_j \colon \fg \to V_j$ denote the subspace projection onto $V_j$ and 
\begin{equation*}
    \pi_j := \exp \circ \,\Pi_j \circ \exp^{-1} \colon G \to G 
\end{equation*}
the corresponding map in the Lie group. We may then reinterpret \eqref{eq: l-fold BCH} as
\begin{equation}\label{eq: proj id}
    \pi_i \circ \phi_{\ell}(\exp(X_1), \dots, \exp(X_{\ell})) = \exp \big( \Pi_i \circ \Phi_{\ell}(X_1, \dots, X_{\ell})\big) \quad \textrm{for $1 \leq i \leq \ell$.}
\end{equation}
Furthermore, given any $X_1, \dots, X_{\ell} \in \fg$, as a consequence of \eqref{eq: graded}, we have 
\begin{equation}\label{eq: Phi property 1} 
    \Pi_{i} \circ \Phi_{\ell}(X_1, \dots, X_{\ell}) = 0 \qquad \textrm{for $1 \leq i \leq \ell - 1$}
\end{equation}
and
\begin{equation}\label{eq: Phi property 2} 
    \Pi_{\ell} \circ \Phi_{\ell}(X_1, \dots, X_{\ell}) = \Phi_{\ell}\big(\Pi_{1}(X_1), \dots, \Pi_{1}(X_{\ell})\big).
\end{equation}
On the other hand, given any $x_1, \dots, x_{\ell} \in G$, by combining \eqref{eq: proj id} with \eqref{eq: Phi property 1} and \eqref{eq: Phi property 2}, we have
\begin{equation}\label{eq: tildePhi property 1} 
    \pi_i \circ \phi_{\ell}(x_1, \dots, x_{\ell}) = e \qquad \textrm{for $1 \leq i \leq \ell - 1$}
\end{equation}
and 
\begin{equation}\label{eq: tildePhi property 2} 
    \pi_{\ell} \circ \phi_{\ell}(x_1, \dots, x_{\ell}) = \phi_{\ell}\big(\pi_{1}(x_1), \dots, \pi_{1}(x_{\ell})\big).
\end{equation}
Similarly, the map $\Pi_1 \circ \exp^{-1} \colon G \to V_1$ is a group homomorphism in the sense that
\begin{equation}\label{eq: pi homomorphism}
     \Pi_1 \circ \exp^{-1}(x \bdot y) = \Pi_1 \circ \exp^{-1}(x) +  \Pi_1 \circ \exp^{-1}(y) \qquad \textrm{for all $x$, $y \in G$,}
\end{equation}
where the right-hand sum is in terms of vector addition. 

\begin{proof}[Proof (of Lemma~\ref{lem: gen test})] One direction is clear and so we assume that $\Pi_1 \circ \exp^{-1} (S)$ generates $V_1$. We aim to show that $S$ generates $G$.
As in \eqref{eq: subgroup gen}, let $\langle S \rangle$ denote the subgroup of $G$ generated by $S$. 

We first claim that for any $x = \exp(X) \in G$ with $X \in V_1$, there exists some $\bg(x) \in G$ such that 
\begin{equation}\label{eq: g function}
    \bg(x) \in \langle  S \rangle \quad \textrm{and} \quad \pi_{1}(\bg(x)) = x. 
\end{equation}
Indeed, from our hypothesis on $S$ there exists a finite sequence of elements 
\begin{equation*}
   s_1, \dots, s_k \in S \quad \text{such that} \quad X = \Pi_1 \circ \exp^{-1}(s_1) + \dots + \Pi_1 \circ \exp^{-1}(s_k). 
\end{equation*}
If we define $\bg(x) := s_1 \bdot \dots \bdot s_k$, then clearly $\bg(x) \in \langle S \rangle$ whilst, by \eqref{eq: pi homomorphism}, we also have
\begin{equation*}
\Pi_1 \circ \exp^{-1}(\bg(x)) = \Pi_1 \circ \exp^{-1}(s_1) + \dots + \Pi_1 \circ \exp^{-1}(s_k) = X,
\end{equation*} 
which immediately implies \eqref{eq: g function}. We therefore obtain a function $\bg \colon \exp(V_1) \to G$ satisfying \eqref{eq: g function}. This function is not uniquely defined, but for our purposes it suffices to work with \textit{some} such $\bg$.\footnote{We could easily stipulate additional conditions to ensure $\bg$ is uniquely defined and thus avoid arbitrary choices in the definition.}

Assuming $G$ is an $m$-step group, we now use induction to prove that
\begin{equation}\label{eq: horiz sphere 2}
  G_{\ell} :=  \Big\{ \exp(Y) : Y \in \bigoplus_{i = \ell}^m V_i \Big\} \subseteq \langle S \rangle
\end{equation}
for all $1 \leq \ell \leq m+1$, where $G_{m+1}$ is interpreted as $\{0\}$. For $\ell = 1$, the above statement becomes $G = \langle  S \rangle$, which is precisely the content of the lemma. 

We take $\ell = m+1$ as the base of the induction, in which case \eqref{eq: horiz sphere 2} is trivial. Let $2 \leq \ell \leq m+1$ and suppose, by way of induction hypothesis, that $G_{\ell} \subseteq \langle S \rangle$. To complete the argument, it suffices to show $G_{\ell-1} \subseteq \langle  S \rangle$.

Fix $y \in G_{\ell -1 }$ so that 
\begin{equation*}
    y = \exp(\sum_{i=\ell-1}^m Y_i) \qquad \textrm{for some $Y_i \in V_i$, $\ell-1 \leq i \leq m$.}
\end{equation*} 
Since $G$ is stratified, we can find $X_1, \dots, X_{\ell-1} \in V_1$ such that
\begin{equation*}
\Phi_{\ell-1}(X_1, \dots, X_{\ell-1}) = Y_{\ell-1}.    
\end{equation*}
 Let $x_j := \exp (X_j)$ for $1 \leq j \leq \ell-1$. It follows from \eqref{eq: tildePhi property 2} and \eqref{eq: g function} that
\begin{align*}
   \pi_{{\ell-1}}\big(\phi_{\ell-1}(\bg(x_1), \dots, \bg(x_{\ell-1}))\big) &= \phi_{\ell-1}\big(\pi_{1}(\bg(x_1)), \dots, \pi_{1}(\bg(x_{\ell-1}))\big) \\
   &= \phi_{\ell-1}(x_1, \dots, x_{\ell - 1}) \\
   &= \exp(Y_{\ell-1}),
\end{align*}
where the last step is due to \eqref{eq: proj id} and \eqref{eq: Phi property 2}. On the other hand, from \eqref{eq: tildePhi property 1} we have
\begin{equation*}
   \pi_{i}\big(\phi_{\ell-1}(\bg(x_1), \dots, \bg(x_{\ell-1}))\big) = e \qquad \textrm{for $1 \leq i \leq \ell - 2$.}
\end{equation*}
Consequently, we may write
\begin{equation*}
    z := \phi_{\ell-1}(\bg(x_1), \dots, \bg(x_{\ell-1})) = \exp\Big( Y_{\ell-1} + \sum_{i = \ell}^m Z_{i}\Big) 
\end{equation*}
for some $Z_i \in V_i$ for $\ell \leq i \leq m$. In view of \eqref{eq: g function} and the definition of $z$ in terms of commutators of the $\bg(x_j)$, we have 
$z \in \langle S \rangle$. 

By the Baker--Campbell--Hausdorff formula, there exist polynomial mappings 
\begin{equation*}
    P_{z, i} \colon \bigoplus_{j = \ell}^m  V_j \mapsto V_i 
\end{equation*}
 such that if $u = \exp(U) := \exp( \sum_{i = \ell}^m U_i) \in G_{\ell}$ with $U_i = \Pi_i U \in V_i$, then 
\begin{equation*}
  \pi_i(u \bdot z) = \exp( U_i + P_{z, i}(U)) \qquad \textrm{for $\ell \leq i \leq m$}
\end{equation*}
where $P_{z, i}$ depends only $U_{\ell}, \dots, U_{i-1}$ and $z$. In particular, $P_{z, i}(U)$ is independent of $U_i, \dots, U_m$ and so the polynomial $P_{z, \ell}$ is constant as a function of $U$ (in fact, $P_{z, \ell}(U) = Z_{\ell}$). On the other hand, the remaining projections are given by
\begin{equation*}
    \pi_{i}(u \bdot z) = e \quad \textrm{for $1 \leq i \leq \ell-2$} \quad \textrm{and} \quad \pi_{{\ell-1}}(u \bdot z) = \exp(Y_{\ell-1}) 
\end{equation*}

For any $u \in G_{\ell}$ as above, the induction hypothesis implies that $u \bdot z \in \langle  S \rangle$. In view of the dependence properties of the $P_{z,i}$, it is possible to inductively choose the $U_i$ so that
\begin{equation*}
   Y_i = U_i +  P_{z,i}(U) \qquad \textrm{for $\ell \leq i \leq m$.}
\end{equation*}
Thus, from the preceding observations,  $y = u \bdot z \in \langle S \rangle$ and we conclude that $G_{\ell-1} \subseteq \langle  S \rangle$. This closes the induction and completes the proof. 
\end{proof}




\appendix

\section{Existence of the Littlewood--Paley decomposition}\label{appendix: Littlewood--Paley}

Here we provide a proof of Proposition~\ref{prop: Littlewood--Paley}. For this, it is convenient to adopt slightly different notation from that used in the rest of the paper: given $f \in C(G)$ and a continuous parameter $t > 0$ we shall write $f_t := t^{-Q} f \circ \delta_{t^{-1}}$, so that the notation $f_k$ used previously in the paper corresponds to $f_{2^k}$. 

 Proposition~\ref{prop: Littlewood--Paley} follows from a basic result on $L^p$ approximate identities. Consider $\phi \in C_c^{\infty}(G)$ satisfying 
 \begin{equation}\label{eq: app 1}
    \int_G \phi = 1.  
 \end{equation}
 Given any $f \in C_c(G)$, it follows that 
\begin{equation}\label{eq: app 2a}
    \|f \ast \phi_t - f\|_{L^p(G)} \to 0 \quad \textrm{as $t \to 0_+$} \qquad \textrm{for all $1 \leq p \leq \infty$} 
\end{equation}
and
\begin{equation}\label{eq: app 2b}
    \|f \ast \phi_t\|_{L^p(G)} \to 0 \quad \textrm{as $t \to \infty$}  \qquad \textrm{for all $1 < p \leq \infty$;}
\end{equation}
the standard proofs are left to the reader (also see \cite[Proposition 1.20]{FS_book}). 

\begin{proof}[Proof (Proposition~\ref{prop: Littlewood--Paley})] 

Suppose $\phi \in C_c^{\infty}(G)$ satisfies \eqref{eq: app 1} as above. By \eqref{eq: app 2a} and \eqref{eq: app 2b}, we have
\begin{equation*}
    f = \lim_{K \to \infty} f \ast \phi_{2^{-K}} - f \ast \phi_{2^K}
\end{equation*}
and so, by the fundamental theorem of calculus,
\begin{equation}\label{eq: app 3}
    f = - \lim_{K \to \infty} \int_{2^{-K}}^{2^K} f \ast \Big( \frac{\partial \phi_t}{\partial t} \Big) \,\ud t = -\sum_{k \in \Z} f \ast \Big( \int_{2^k}^{2^{k+1}}  \frac{\partial \phi_t}{\partial t}   \,\ud t \Big),
\end{equation}
where in each case the convergence holds in $L^p(G)$ for $1 < p \leq \infty$. A computation shows
\begin{equation*}
    \frac{\partial \phi_t}{\partial t}(x) = -t^{-1} h_t(x) \qquad \textrm{for some $h \in C^{\infty}_c(G)$.}
\end{equation*}
Moreover, if we define
\begin{equation}\label{eq: app 4}
    \psi(x) := \int_1^2 h_t(x)\,\frac{\ud t}{t},
\end{equation}
then, by a simple change of variables, 
\begin{equation}\label{eq: app 5}
    -\int_{2^k}^{2^{k+1}}  \frac{\partial \phi_t}{\partial t}(x)  \,\ud t = \int_{2^k}^{2^{k+1}}  h_t(x)  \,\frac{\ud t}{t} = \psi_{2^k}(x).
\end{equation}
Combining \eqref{eq: app 3} and \eqref{eq: app 5}, we see that \eqref{eq: fourier proj 2} holds for $\psi$ as defined in \eqref{eq: app 4}. 

It remains to show $\psi$ is of mean zero. Clearly, it suffices to show the same property hold for the function $h$. However, since
\begin{equation*}
\int_G h(x) \,\ud x = - t \frac{\partial}{\partial t} \int_G \phi_t(x) \,\ud x \Big|_{t = 1},
\end{equation*}
the mean zero property for $h$ is an immediate consequence of \eqref{eq: app 1}. 
\end{proof}




\bibliography{Reference}
\bibliographystyle{amsplain}

\end{document}